\documentclass[12pt]{amsart}
\usepackage{amsmath,amsthm,amssymb,amsfonts,dsfont}
\usepackage{graphicx}
\usepackage{color}
\usepackage[utf8]{inputenc}

\newtheorem{theorem}{Theorem}[section]
\newtheorem{lemma}[theorem]{Lemma}

\theoremstyle{definition}
\newtheorem{definition}[theorem]{Definition}

\newtheorem{example}[theorem]{Example}

\numberwithin{equation}{section}
\def\llll{\longrightarrow}

\newcommand{\C}{{\mathbb {C}}}
\newcommand{\N}{{\mathbb {N}}}

\newcommand{\R}{{\mathbb {R}}}

\newcommand{\1}{{\mathds{1}}}

\newcommand{\tnorm}{|\hskip-.07em\|}
\def\nnn{\tnorm \ \tnorm}

\def\supp{{\rm supp \;}}

\def\sep{{ \ \  }}
\def\sem{{\ \ \ \  }}
\def\seg{{\ \ \ \  \ \  }}

\title[Bishop-Phelps-Bollob{\'a}s  property] {Bishop-Phelps-Bollob{\'a}s  property  for positive \\ operators  between classical Banach spaces
}

\author[M.D. Acosta]{Mar\'{\i}a D. Acosta}
\address{Universidad de Granada, Facultad de Ciencias,
	Departamento de An\'{a}lisis Matem\'{a}tico, 18071 Granada, Spain}
\email{dacosta@ugr.es}
\author[M. Soleimani]{Maryam Soleimani-Mourchehkhorti}
\address{School of Mathematics, Institute for Research in Fundamental Sciences (IPM), P.O. Box: 19395-5746, Tehran, Iran}
\email{m-soleimani85@ipm.ir}

\thanks{The  first  author was  supported  by Junta de Andaluc\'{\i}a grant  FQM--185  and also by Spanish MINECO/FEDER grants  MTM2015-65020-P and PGC2018-093794-B-I00. The second author was   supported by a grant from IPM}

\keywords{Banach space, operator, Bishop-Phelps-Bollob{\'a}s  theorem, Bishop-Phelps-Bollob{\'a}s  property.}

\begin{document}

\subjclass[2010]{Primary 46B04; Secondary 47B99.}

  {\large

\dedicatory{Dedicated to the memory of Victor Lomonosov}
\begin{abstract}
	We prove that the class of  positive operators from   $L_\infty (\mu)$ to  $L_1 (\nu)$ has the Bishop-Phelps-Bollobás property for any positive measures $\mu$ and $\nu$. The same result also holds for the  pair $(c_0, \ell_1)$. We also provide an example showing that not every pair of Banach lattices satisfies the Bishop-Phelps-Bollobás property for  positive  operators.
\end{abstract}

\maketitle

   \section{Introduction}
   
   In 1961 Bishop and Phelps proved that for any Banach space the set of (bounded and linear) functionals attaining their  norms is norm dense in the topological dual space \cite{BP}. In 1970 Bollobás gave some quantified version of that result \cite{Bol}.    In order to state such result we recall the following notation.   By     $B_X$, $S_X$ and $X^*$   we denote  the closed unit ball,  the unit sphere
   and the topological dual of a Banach space $X,$ respectively.  If $X$ and $Y$ are  both real or  both complex Banach spaces, $L(X,Y)$ denotes the space of (bounded linear) operators from $X$ to $Y,$ endowed with its usual operator norm.

   \vskip3mm
   
   {\it Bishop-Phelps-Bollob{\'a}s theorem} (see \cite[Theorem 16.1]{BoDu} or \cite[Corollary 2.4]{CKMMR}). Let $X$ be a Banach space and $0< \varepsilon < 1$. Given $x \in B_X$  and $x^* \in S_{X^*}$ with $\vert 1- x^* (x) \vert < \frac{\varepsilon ^2 }{2}$, there are elements $y \in S_X$ and $y^* \in S_{X^*}$  such that $y^* (y)=1$, $\Vert y-x \Vert < \varepsilon$ and $\Vert y^* - x^* \Vert < \varepsilon $.

   After a period  in which a lot of attention has been  devoted to extending Bishop-Phelps theorem  to operators and   interesting results have been  obtained about that topic (see  \cite{Acs}), in 2008 it was posed the problem of extending  Bishop-Phelps-Bollob{\'a}s theorem for operators. 
   
   In order to state some of these extensions it will be convenient to  recall  the following notion.

   \begin{definition}
   	[{\cite[Definition 1.1]{AAGM}}]
   	\label{def-BPBp}
   	Let $X$  and $Y$ be either real or complex Banach spaces. The pair $(X,Y )$ is said to have the Bishop-Phelps-Bollob{\'a}s property for operators (BPBp) if for every $  0 < \varepsilon  < 1 $  there exists $ 0< \eta (\varepsilon) < \varepsilon $ such that for every $S\in S_{L(X,Y)}$, if $x_0 \in S_X$ satisfies
   	$ \Vert S (x_0) \Vert > 1 - \eta (\varepsilon)$, then
   	there exist an element $u_0 \in S_X$  and an operator $T \in S_{L(X,Y )}$ satisfying the following conditions
   	$$
   	\Vert T (u_0) \Vert =1, \sem \Vert u_0- x_0 \Vert < \varepsilon \seg \text{and}
   	\sem \Vert T-S \Vert < \varepsilon.
   	$$
   \end{definition}
   
   If  $X$ and $Y$ are Banach spaces, it  is known that the  pair $(X,Y)$ has the Bishop-Phelps-Bollobás property in the following cases:
   \begin{itemize}
   	\item $X$ and  $Y$ are finite-dimensional spaces \cite[Proposition 2.4]{AAGM}.
   	\item $X$  is any Banach space and $Y$ has the property $\beta $ (of Lindenstrauss) \cite[Theorem 2.2]{AAGM}. The spaces $c_0 $ and $\ell_\infty $ have property $\beta$.
   	\item $X$  is uniformly convex and $Y$ is any Banach space  (\cite[Theorem 2.2]{ABGt} or \cite[Theorem 3.1]{KLc}).
   	\item $X=\ell_1$ and $Y$ has the  approximate hyperplane series property \cite[Theorem 4.1]{AAGM}. For instance, finite-dimensional spaces, uniformly convex spaces,   $C(K)$ and  $L_1 (\mu)$ have the approximate hyperplane series property.
   	\item $X=L_1 (\mu)$ and $Y$ has the Radon-Nikod\'{y}m property and the approximate hyperplane series property, whenever  $\mu$ is any  $\sigma$-finite  measure \cite[Theorem 2.2]{CK1} (see also \cite[Theorem 2.3]{ABGj}).
   	\item $X=L_1(\mu)$ and $Y= L_1 (\nu)$,  for any positive measures $\mu$ and $\nu$ 
   	\cite[Theorem 3.1]{CKLj}.
   	\item $X=L_1(\mu)$ and $Y= L_\infty (\nu)$,  for any positive measure $\mu$ and any localizable positive measure  $\nu$ 
   	\cite[Theorem 4.1]{CKLj} (see also  \cite{ACGM}).
   	\item $X=C(K)$ and $Y= C(S)$ in the real case, where $K$ and $S$ are compact Hausdorff topological spaces \cite[Theorem 2.5]{ABCC}.
   	%\item $X=c_0$ and $Y$ is a uniformly convex Banach space \cite[Corollary 2.6]{KiI}.
   	\item $X=C(K)$ and $Y$ is a uniformly convex Banach space, in the real case \cite[Theorem 2.2]{KL2} (see also  \cite[Corollary 2.6]{KiI} and  \cite[Theorem 5]{KLL}).
   	\item $X=C_0(L)$, for any
   	locally compact Hausdorff topological space $L$ and  $Y$ is a  $\C$-uniformly convex space, in the complex case \cite[Theorem 2.4]{AcB}.   As a consequence, the pair  $(C_0(L),L_p (\mu))$ has the BPBp for any positive measure $\mu$ and $1 \le p < +\infty$. 
   	\item $X=\ell_\infty ^n$ and $Y =L_1 (\mu) $ for any positive integer $n$ and any  positive measure $\mu$  \cite[Corollary 4.5]{AD}  (see also \cite[Theorem 3.3]{AD}, \cite[Theorem 3.3]{ADS} and \cite[Theorem 2.9]{ABGp}).
   	\item $X$ is an Asplund space and  $\mathcal{A} \subset C(K)$ is a  uniform algebra 
   	\cite[Theorem 3.6]{CGK} (see also \cite[Corollary 2.5]{ACK}).
   \end{itemize}
   The paper \cite{Acbc}  contains a survey with most of the results known  about  the Bishop-Phelps-Bollobás property  for operators.
   
   In this short note we introduce a version of Bishop-Phelps-Bollobás property  for positive operators between Banach lattices (see Definition \ref{def-BPBp-pos}).
   The only difference between this property and the previous one is that we assume that the operators appearing in Definition \ref{def-BPBp} are positive. 
   In Section 2 we prove  that the pair $(L_\infty(\mu), L_1 (\nu)) $  has the Bishop-Phelps-Bollobás property for positive operators for any positive measures $\mu$ and $\nu$.  The parallel result for  $(c_0, L_1 (\mu))$  is shown in section 3, for any positive measure $\mu$.  As a consequence,  the subset of positive operators from $ c_0$ to $\ell_1$ satisfies the Bishop-Phelps-Bollobás property. We remark that it is not known  whether   the pairs  $(L_\infty(\mu), L_1 (\nu))$ and  $(c_0, \ell_1)$ satisfy the Bishop-Phelps-Bollobás property for operators in the real case. In both cases the set of norm attaining operators is dense in the space of operators (see \cite[Theorem B]{Scha} for the first case). For the second pair, a necessary condition  on the range space  in order to have the Bishop-Phelps-Bollobás for operators  is known (see \cite[Theorem 3.3]{AD}). We also provide an example showing that not every  pair of Banach lattices satisfies the  Bishop-Phelps-Bollobás property for positive operators.

\section{ Bishop-Phelps-Bollobás property for positive operators for the pair $(L_\infty, L_1)$}

We begin  by recalling some notions and introducing the appropriate notion of  Bishop-Phelps-Bollobás property  for positive operators.
The concepts in the first definition  are standard and can be found, for instance, in \cite{AbAl}.

\begin{definition}
	\label{def-lattice}
	An \textit{ordered vector space} is a real vector space $X$  equipped with a vector space order, that is,  an order relation $\le$  that is compatible with the algebraic structure of $X$.   An ordered vector space is called a  \textit{Riesz space}  if every pair of vectors has a least upper bound  and a greatest lower bound. A norm $\Vert \ \Vert $ on a Riesz space   $X$ is said  to be a   \textit{lattice norm} whenever $|x| \leq |y|$ implies $\Vert x\Vert  \le \Vert y \Vert  $. \textit{A normed  Riesz space} is a  Riesz space equipped  with a lattice norm. A normed Riesz space whose norm is complete  is called a \textit{Banach lattice}.
	
	An operator $T: X \to Y$ between two ordered vector spaces is called \textit{positive} if $x \geq 0$ implies $Tx \geq 0$.
\end{definition}

We remark that throughout this paper by operator we mean a linear mapping.  Recall that every positive operator from a Banach lattice to a normed Riesz space is continuous \cite[Theorem 4.3]{AlBur}.

\begin{definition}
	\label{def-BPBp-pos}
	Let $X$  and $Y$ be    Banach lattices. The pair $(X,Y )$ is said to have the {\it Bishop-Phelps-Bollob{\'a}s property for  positive operators}  if for every $  0 < \varepsilon  < 1 $  there exists $ 0< \eta (\varepsilon) < \varepsilon $ such that for every $S\in S_{L(X,Y)}$, such that $S \ge 0$,  if $x_0 \in S_X$ satisfies
%	\linebreak[4]
	$ \Vert S (x_0) \Vert > 1 - \eta (\varepsilon)$, then
	there exist an element $u_0 \in S_X$  and a positive  operator $T \in S_{L(X,Y )}$ satisfying the following conditions
	$$
	\Vert T (u_0) \Vert =1, \sem \Vert u_0- x_0 \Vert < \varepsilon \seg \text{and}
	\sem \Vert T-S \Vert < \varepsilon.
	$$
\end{definition}

Let $ (\Omega, \mu) $   be a measure space. We denote by $L_\infty (\mu)$ the space of real valued  measurable essentially bounded functions on $\Omega$  and by $\1$  the constant function equal to $1$ on $\Omega$. Since an element $f$  in  $B_{L_\infty (\mu)}$ satisfies that $ \vert f \vert \le  \1$ a.e.,  it is clear that a positive operator from $L_\infty(\mu)$ to any other Banach lattice satisfies the next result.

\begin{lemma}
	\label{le-pos}
Let $\mu$ and $\nu$ be positive measures and  $T$ a positive operator from 	$ L_\infty (\mu)$ to  $L_1 (\nu)$. Then $\Vert T \Vert = \Vert T(\1)\Vert_1$.
\end{lemma}

It is trivially satisfied that   $\Vert f+g \Vert _1 =  \Vert f-g \Vert _1$  for any positive   integrable  functions $f$ and $g$  with disjoint supports.
Next result shows that in case that two  functions satisfy the previous assumption and the quantities  $\Vert f+g \Vert _1$ and $\Vert f-g \Vert _1$ are close, then both functions can be approximated  by positive functions with disjoint supports.

\begin{lemma}
	\label{le-dis-supp}
	Let $(\Omega, \mu)$  be a  measure space, $0 < \varepsilon < \dfrac{1}{5}$  and $f_1, f_2 \in L_1 (\mu)$  be  positive functions such that
	$$
	\Vert f_1 + f_2 \Vert _1 \le 1 \sem \text{\rm and} \sem 1- \varepsilon ^2  \le 	\Vert f_1 - f_2 \Vert _1 .
	$$
	Then there are two positive functions $g_1$ and $g_2$  with disjoint supports  in $L_1(\mu)$  and also satisfying that
	$$
	\Vert g_1 + g_2 \Vert _1 =1 	 \sem \text{\rm and} \sem \Vert g_i- f_i \Vert _1 < 7 \varepsilon \sep \text{for}\sep i=1,2  .
	$$
\end{lemma}
\begin{proof}
	We define the set  $W$ given by
	$$
	W= \{ t \in \Omega: \vert f_1 (t)- f_2 (t) \vert  \le (1- \varepsilon) \bigl( f_1 (t) + f_2 (t) \bigr )\}.
	$$
	Clearly $W$ is a measurable subset of $\Omega$. We have that 
	\begin{align*}
	1 - \varepsilon ^2	&  \le  \Vert f_1 - f_2 \Vert_1 	 \\
	& =  \int _\Omega \vert f_1 - f_2 \vert \ d \mu \\
	&  =  \int _W\vert f_1 - f_2 \vert \; d \mu  + \int _{\Omega\backslash W}  \vert f_1 - f_2 \vert \ d \mu \\
	&   \le   (1- \varepsilon) \int _W ( f_1 + f_2 ) \; d \mu  + \int _{\Omega\backslash W}  ( f_1 + f_2 ) \ d \mu    \\
	&  \le 1-  \varepsilon \int _W ( f_1 + f_2 ) \; d \mu .
	\end{align*}
	As a consequence
	\begin{equation}
	\label{int-W-small}
	\int _W (f_1 + f_2 ) \; d \mu \le \varepsilon.
	\end{equation}
	Now we define the sets given by
	$$
	G_1 = \Omega \backslash W \cap \{ t \in \Omega: f_1 (t) > f_2 (t) \} \sem \text {and} \sem  G_2 = \Omega \backslash W \cap \{ t \in \Omega: f_2 (t) > f_1 (t) \}.
	$$
	Clearly $G_1$ and $G_2$ are measurable subsets and it is satisfied that
	$$
	(f_1- f_2) \chi _{G_1} = \vert f_1 - f_2 \vert \chi _{G_1}  > (1- \varepsilon) (f_1+f_2) \chi _{G_1}.
	$$
	So $ f_2 \chi_{G_1} \le (2 - \varepsilon)  f_2 \chi_{G_1}  \le \varepsilon f_1  \chi_{G_1}   $ and 
	\begin{equation}
	\label{int-G1-f2-small}
	\int _{G_1} f _2 \ d \mu \le  \int _{G_1}  \varepsilon f _1 \ d \mu \le  \varepsilon.
	\end{equation}
	By using the same argument with the function $f_1$ we obtain that 
	\begin{equation}
	\label{int-G2-f1-small}
	\int _{G_2} f _1 \ d \mu \le   \varepsilon.
	\end{equation}
	Since the subsets $W$, $G_1$ and $G_2$ are a partition of $\Omega $, in view of \eqref{int-W-small}	 and \eqref{int-G2-f1-small} we deduce that
	\begin{align}
	\label{f1-close-f1G1}
	\Vert f_1 - f_1 \chi _{ G_1} \Vert _1	& = \Vert f_1 \chi _{  W \cup G_2} \Vert _1 
	\nonumber	 \\
	& =    \Vert f_1 \chi _  W \Vert _1 + \Vert f_1 \chi _{  G_2} \Vert _1 \\
	& = \int _{W} f_1 \ d \mu  +   \int _{G_2} f_1 \ d \mu  \nonumber\\
	&   \le  2 \varepsilon. \nonumber
	\end{align}
	Since $f_1$ and $f_2$ satisfy the same  conditions we also have that
	\begin{equation}
	\label{f2-close-f2G2}
	\Vert f_2 - f_2 \chi _{ G_2} \Vert _1 \le  2 \varepsilon.
	\end{equation}
	By using that $f_1$ and $f_2$ are positive functions we deduce that
	\begin{align}
	\label{g1G1-g2G2-big}
	\Vert f_1 \chi _{G_1} + f_2 \chi_{ G_2} \Vert _1 & =  \Vert f_1 + f_2  \Vert _1  	 -  \Vert f_1  - f_1\chi _{G_1}  \Vert _1 -  \Vert f_2  - f_2\chi _{G_2}  \Vert _1
	\nonumber\\
	& \ge  \Vert f_1  - f_2  \Vert _1 - 4 \varepsilon  \sem \text{(by \eqref{f1-close-f1G1} and \eqref{f2-close-f2G2})} \\
	& \ge  1- \varepsilon ^2 - 4 \varepsilon   \nonumber \\
	&   > 1 - 5 \varepsilon > 0. \nonumber
	\end{align}
	Now we define the functions $g_1$ and $g_2$ by
	$$
	g_i= \frac{f_i \chi_{G_i}}{ \Vert f_1 \chi_{G_1} + f_2 \chi_{G_2} \Vert _1 }, \sem i=1,2.
	$$
	It is clear that $g_i \in L_1 (\mu)$  for $i=1,2$ and they are positive functions with disjoint supports.  It is also clear that $\Vert g_1 + g_2 \Vert _1 =1$.  
	
	Since $f_1$ and $f_2$ are positive functions we have that
	$\Vert f_1 \chi_{G_1} + f_2 \chi _{G_2} \Vert _1 \le  \Vert f_1  + f_2  \Vert _1 \le  1$, so for $i=1,2$ we have that 
	\begin{align}
	\label{comp-f1G1-normal}
	\Vert f_i \chi_{G_i} \Vert _1  \biggl \vert \frac{1}{ \Vert f_1 \chi_{G_1} + f_2 \chi _{G_2} \Vert _1 }  -1 \biggr \vert  & =  \Vert f_i \chi_{G_i} \Vert _1  \biggl (  \frac{1}{ \Vert f_1 \chi_{G_1} + f_2 \chi _{G_2} \Vert _1 } - 1  \biggr)  
	\nonumber\\
	& =    \frac{  \Vert f_i \chi_{G_i} \Vert _1  \bigl (   1-  \Vert f_1 \chi_{G_1} + f_2 \chi _{G_2} \Vert _1 \bigr)  } { \Vert f_1 \chi_{G_1} + f_2 \chi _{G_2} \Vert _1 }\\
	&\le  1- \Vert f_1 \chi_{G_1} + f_2 \chi _{G_2} \Vert _1  \nonumber.
	\end{align}
	For $i=1,2$ we estimate the distance from  $g_i$ to $f_i$ as follows
	\begin{align*}
	\Vert g_i - f_i \Vert_1	 & =  \Bigl \Vert  \frac{f_i \chi_{G_i}}{ \Vert f_1 \chi_{G_1} + f_2 \chi_{G_2} \Vert _1 } - f_i \Bigr \Vert _1  \\
	&\le  \Bigl \Vert  \frac{f_i \chi_{G_i}}{ \Vert f_1 \chi_{G_1} + f_2 \chi_{G_2} \Vert _1 }  - f_i \chi_{G_i}  \Bigr \Vert _1 +	 \Vert  f_i \chi_{G_i} - f_i  \Vert _1   \\
	&\le  1 -   \Vert f_1 \chi_{G_1} + f_2 \chi_{G_2} \Vert _1  + 2 \varepsilon  \sem \text{(by \eqref{comp-f1G1-normal},  \eqref{f1-close-f1G1} and \eqref{f2-close-f2G2})}\\
	&<  7 \varepsilon \seg \text{(by \eqref{g1G1-g2G2-big})}.
	\end{align*}
\end{proof}

\begin{theorem}
	\label{teo-BPBp-L-infty-L1}
	For any positive measures 	 $\mu$ and $\nu$, the pair $(L_\infty (\mu), L_1 (\nu))$ has the Bishop-Phelps-Bollobás property for positive operators. 
	
	Moreover,  in Definition \ref{def-BPBp-pos},  if the function $f_0$ where the operator $S$ is close to attain its norm  is positive, then the function $f_1$ where $T$ attains its norm is also positive.
\end{theorem}
\begin{proof}
	Assume that $(\Omega_1, \mu)$ is a measure space.
	Let $ 0 < \varepsilon < 1,$  $f_0 \in S_{L_\infty (\mu)} , S\in  S_{ L( L_\infty (\mu), L_1 (\nu))}	$ and assume that $S$ is a positive operator satisfying that 
	$$
	\Vert S(f_0) \Vert _1 > 1 - \eta^2,
	$$
	where $\eta = \bigl(\frac{ \varepsilon}{58}\bigr) ^2$.
	We define the sets $A,B$ and $C$ given by
	$$
	A=\{ t \in \Omega_1 : -1 \le f_0(t) < -1 + \eta \}, \sep B=\{ t \in \Omega_1 :  1- \eta <  f_0(t) \le 1 \} \sep 
	$$
	%\text{and}
	and
	$$
	\sep C=\{ t \in \Omega_1 : \vert f_0(t) \vert \le 1 - \eta \}. 
	$$
	By using that $S$ is a positive operator we obtain that
	\begin{align*}
	1-\eta ^2 	& <  \Vert S (f_0) \Vert _1  	 \\
	& =  \Vert S (f_0 \chi_A + f_0 \chi_B + f_0 \chi_C) \Vert _1   \\
	& \le    \Vert S (f_0 \chi_A ) \Vert _1 +  \Vert S (  f_0\chi_B ) \Vert _1 +  \Vert S ( f_0 \chi_C) \Vert _1     \\
	&    \le    \Vert S ( \chi_A ) \Vert _1 +  \Vert S (  \chi_B ) \Vert _1 +  (1- \eta )  \Vert S (  \chi_C) \Vert _1      \\
	& \le 1 - \eta   \Vert S (  \chi_C) \Vert _1  .
	\end{align*}
	Hence
	$
	\Vert S( \chi_C) \Vert _1 \le \eta .
	$
	By using again that $S$ is positive we deduce that
	\begin{equation}
	\label{SfC-small}
	\Vert S( f\chi_C) \Vert _1 \le   \Vert S( \chi_C) \Vert _1   \le  \eta , \sem  \forall f \in B_{ L_\infty (\mu)} .
	\end{equation}
	On the other hand it is trivially satisfied that 
	$$
	\Vert f_0 \chi_A + \chi_A \Vert _\infty \le \eta \sem \text{and} \sem  \Vert f_0 \chi_B - \chi_B \Vert _\infty \le \eta,
	$$ 
	so 
	\begin{equation}
	\label{Sf0AB-approx}
	\Vert S( f_0\chi_A +\chi_A)\Vert _1 \le    \eta   \sem \text{and} \sem  \Vert S( f_0\chi_B -\chi_B)\Vert _1 \le    \eta  .
	\end{equation}
	By using  the assumption  we obtain that 
	\begin{align*}
	1-\eta ^2 	& <  \Vert S (f_0) \Vert _1  	 \\
	& \le \Vert S (f_0 \chi_A + f_0 \chi_B  ) \Vert _1    +  \Vert S(f_0 \chi_C) \Vert _1  \\
	& \le    \Vert S (f_0 \chi_A  + \chi_A ) \Vert _1 +  \Vert S ( \chi_B - \chi_A) \Vert _1 +  \Vert S ( f_0 \chi_B - \chi_B) \Vert _1 + \Vert S(f_0 \chi_C )  \Vert _1   \\
	&    \le   \Vert S ( \chi_B - \chi_A) \Vert _1 + 3 \eta \sem \text{(by \eqref{Sf0AB-approx} and \eqref{SfC-small})}.
	\end{align*}
	As a consequence
	\begin{equation}
	\label{SB-A-small}
	\Vert S( \chi_B - \chi_A) \Vert _1  \ge   1- 4 \eta  .
	\end{equation}
	Since $S(\chi_A)$ and $S(\chi_B)$ are positive  functions and $\Vert S (\chi_A) + S (\chi_B)\Vert _1 \le 1$ we can apply Lemma \ref{le-dis-supp} and so there are two positive functions  $g_1$ and $g_2$ in $L_1 (\nu)$  satisfying the following conditions
	$$
	\Vert g_1- S(\chi_A) \Vert _1 <  14 \sqrt{\eta} = \frac{7 \varepsilon}{29},  \sem   \Vert g_2- S(\chi_B) \Vert _1  <    \frac{7 \varepsilon}{29} ,
	$$
	$$
	\supp g_1 \cap  \supp g_2 = \varnothing  \sem \text{and} \sem \Vert g_1+g_2 \Vert _1  =1.
	$$

	Assume that $\nu$ is a measure on $\Omega _2$. We obtain that
	\begin{equation}
	\label{S-A}
	\Vert  S (\chi_A) \chi _{ \Omega _2 \backslash \supp g_1} \Vert _1 =  \Vert  (g_1 - S(\chi_A))  \chi _{ \Omega _2 \backslash \supp g_1} \Vert _1 \le \Vert g_1 -  S (\chi_A)  \Vert _1 < \frac{7 \varepsilon}{29}
	\end{equation}
	and also
	\begin{equation}
	\label{S-B}
	\Vert  S (\chi_B) \chi _{ \Omega _2 \backslash \supp g_2} \Vert _1  < \frac{7 \varepsilon}{29}.
	\end{equation}
	
	Now we define the operator $V: L_\infty (\mu) \llll L_1 (\nu)$ as follows
	$$
	V(f)= S(f \chi_A) \chi _{\supp g_1} +  S(f \chi_B) \chi _{\supp g_2} \seg (f \in L_\infty (\mu)).
	$$
	Clearly $V$ is well defined and it is a positive operator since $S \ge 0$. So
	$$
	\Vert  V \Vert = \Vert V (\1)\Vert _1= \Vert S(\chi_A) \chi _{\supp g_1} +  S(\chi_B) \chi _{\supp g_2}  ) \Vert _1 \le \Vert S \Vert = 1.
	$$
	Now we estimate the norm of $V-S$. If $f \in B_{ L_\infty (\mu)}$ then we have that
	\begin{align*}
	\Vert (V -  S)(f) \Vert _1 & =  \Vert  S(f \chi_A) \chi _{\supp g_1} +  S(f \chi_B) \chi _{\supp g_2}   ) - S(f) \Vert _1 \\
	&  = \Vert  S(f \chi_A) \chi _{\supp g_1} +  S(f \chi_B) \chi _{\supp g_2}   ) - S(f \chi_A)  - S(f \chi_B)  -S(f \chi_C) \Vert _1 \\
	&\le \Vert  S(f \chi_A) \chi _{ \Omega _2 \backslash \supp g_1}  \Vert _1 +  \Vert  S(f \chi_B) \chi _{\Omega _2 \backslash\supp g_2}   )  \Vert _1 + \Vert S(f \chi_C) \Vert _1 \\
	&\le \Vert  S(\chi_A) \chi _{ \Omega _2 \backslash \supp g_1}  \Vert _1 +  \Vert  S( \chi_B) \chi _{\Omega _2 \backslash\supp g_2}   )  \Vert _1 + \Vert S(f \chi_C) \Vert _1 \\
	&<  \frac{ 14 \varepsilon}{29} + \eta  < \frac{\varepsilon}{2} \sem \text{(by \eqref{S-A},  \eqref{S-B} and \eqref{SfC-small})}.
	\end{align*}
	We proved that $\Vert V - S \Vert  < \frac{\varepsilon}{2}$  and so $\Vert V \Vert \ge 1 - \dfrac{\varepsilon}{2} > 0$. 
	Since $f_0 \in S_{ L_\infty (\mu)}$ the function $f_1$ given by $f_1 = \chi_B- \chi_A + f_0 \chi _C \in  S_{ L_\infty (\mu)}$ and satisfies
	$$
	\Vert f_1 - f_0 \Vert _\infty \le \eta < \varepsilon.
	$$
	Since $g_1$ and $g_2$ have disjoint supports we also have that
	\begin{align*}
		\Vert V (f_1) \Vert _1 & = \Vert S ( -\chi_A) \chi_{\supp g_1} +  S ( \chi_B) \chi_{\supp g_2} \Vert _1  \\
		& =  \Vert S ( \chi_A) \chi_{\supp g_1} +   S ( \chi_B) \chi_{\supp g_2} \Vert _1  \\
		&  = \Vert V(\1) \Vert _1 = \Vert V \Vert  . \\
	\end{align*}
	If we take $T= \frac{V}{\Vert V \Vert}$, the operator $T \in S_{L(L_\infty (\mu), L_1 (\nu)}$, is a positive operator, attains its norm at $f_1$ and satisfies that
	$$
	\Vert T - S \Vert \le  \Vert T - V \Vert +\Vert  V-S \Vert  = \bigl\vert 1 - \Vert V \Vert  \bigr \vert   +\Vert  V-S \Vert \le 2 \Vert V - S \Vert < \varepsilon.
	$$
	We proved that the pair $(L_\infty (\mu), L_1 (\nu))$  has the Bishop-Phelps-Bollobás property for positive operators. In case that $f_0 \ge 0$ the function $f_1$ also satisfies the same condition.
\end{proof}

   \section{A positive result for the Bishop-Phelps-Bollobás property for positive operators for  $(c_0, L_1 )$.}  		
   
   \begin{theorem}
   	\label{teo-BPBp-c0-L1}
   	For any positive measure 	 $\mu$, the pair $(c_0, L_1 (\mu))$ has the Bishop-Phelps-Bollobás property for positive operators. 
   	
   	Moreover,  in Definition \ref{def-BPBp-pos},  if the element  $x_0$  is positive, then the elemnent $u_0$ where $T$ attains its norm is also positive.
   \end{theorem}
   \begin{proof}
   	The proof of this result is similar  to the proof of Theorem \ref{teo-BPBp-L-infty-L1}. In any case we include it for the sake of completeness.
   	Throughout this proof  we denote by $\Vert \ \Vert $ the usual norm of $c_0$.

   	Assume that $\Omega$ is the set  such that  $(\Omega, \mu)$ is the measure
   	space considered for  $L_1(\mu)$.
   	Let $ 0 < \varepsilon < 1,$  $x_0 \in S_{c_0},$ $S\in  S_{ L( c_0, L_1 (\mu))}	$ and assume that $S$ is a positive operator satisfying that 
   	$$
   	\Vert S(x_0) \Vert _1 > 1 - \eta^2,
   	$$
   	where $\eta = \bigl(\frac{ \varepsilon}{58}\bigr) ^2$.
   	We define the sets $A,B$ and $C$ given by
   	$$
   	A=\{ k \in \N  : -1 \le x_0(k) < -1 + \eta \}, \sep B=\{ k \in \N :  1- \eta <  x_0(k) \le 1 \} \sep 
   	$$
   	and
   	$$
   	\sep C=\{ k \in \N : \vert x_0(k) \vert \le 1 - \eta \}. 
   	$$
   	Since $x_0\in S_{c_0}$ the sets  $A$  and $B$ are finite and $\{A,B,C\}$ is a partition  of $\N$.
   	
   	For each positive integer $n$ we denote by $C_n = C \cap  \{k \in \N : k \le n\}$, which is a finite subset of $\N$.
   	By using that $S$ is a positive operator in $S_{ L(c_0, L_1 (\mu))}$ we obtain that
   	\begin{align*}
   	1-\eta ^2 	& <  \Vert S (x_0) \Vert _1  	 \\
   	& =  \Vert S (x_0 \chi_A + x_0 \chi_B + x_0 \chi_C) \Vert _1   \\
   	& \le    \Vert S (x_0 \chi_A ) \Vert _1 +  \Vert S (  x_0\chi_B ) \Vert _1 +  \Vert S ( x_0 \chi_C) \Vert _1     \\
   	&    \le    \Vert S ( \chi_A ) \Vert _1 +  \Vert S (  \chi_B ) \Vert _1 +  (1- \eta )   \lim_{n}  \bigl \{ \Vert S (  \chi_{C_n}) \Vert _1 \bigr \}      \\
   	& \le 1 - \eta    \lim_{n}  \bigl \{ \Vert S (  \chi_{C_n}) \Vert _1 \bigr \}  .
   	\end{align*}
   	Hence
   	$
   	 \lim_{n}  \bigl \{ \Vert S (  \chi_{C_n}) \Vert _1 \bigr \}  \le \eta .
   	$
   	Since $S$ is positive we get that 	
   	\begin{equation}
   	\label{SxC-small}
   	\Vert S( x\chi_C) \Vert _1 \le   \lim_{n}  \bigl \{ \Vert S (  \chi_{C_n}) \Vert _1 \bigr \}  \le   \eta , \sem  \forall x \in B_{ c_0}.
   	\end{equation}
   	On the other hand it is trivially satisfied that 
   	$$
   	\Vert x_0 \chi_A + \chi_A \Vert  \le \eta \sem \text{and} \sem  \Vert x_0 \chi_B - \chi_B \Vert  \le \eta,
   	$$ 
   	and so 
   	\begin{equation}
   	\label{Sx0AB-approx}
   	\Vert S( x_0\chi_A +\chi_A)\Vert _1 \le    \eta   \sem \text{and} \sem  \Vert S( x_0\chi_B -\chi_B)\Vert _1 \le    \eta  .
   	\end{equation}
   	In view of the  assumption, since $\{A,B,C\}$ is a partition of $\N$   we obtain that 
   	\begin{align*}
   	1-\eta ^2 	& <  \Vert S (x_0) \Vert _1  	 \\
   	& \le \Vert S (x_0 \chi_A + x_0 \chi_B  ) \Vert _1    +  \Vert S(x_0 \chi_C) \Vert _1  \\
   	& \le    \Vert S (x_0 \chi_A  + \chi_A ) \Vert _1 +  \Vert S ( \chi_B - \chi_A) \Vert _1 +  \Vert S ( x_0 \chi_B - \chi_B) \Vert _1 + \Vert S(x_0 \chi_C )  \Vert _1   \\
   	&    \le   \Vert S ( \chi_B - \chi_A) \Vert _1 + 3 \eta \sem \text{(by \eqref{Sx0AB-approx} and \eqref{SxC-small})}.
   	\end{align*}
   	Hence
   	\begin{equation}
   	\label{c0-SB-A-big}
   	\Vert S( \chi_B - \chi_A) \Vert _1  \ge   1- 4 \eta  .
   	\end{equation}
   	Now we can apply  Lemma \ref{le-dis-supp}  to the positive  functions  $S(\chi_A)$ and $S(\chi_B)$  since $\Vert S (\chi_A) + S (\chi_B)\Vert _1 \le \Vert S \Vert =1 $. So  there exist  two positive functions  $g_1$ and $g_2$ in $L_1 (\mu)$  satisfying the following conditions
   	$$
   	\Vert g_1- S(\chi_A) \Vert _1 < \frac{7 \varepsilon}{29},  \sem   \Vert g_2- S(\chi_B) \Vert _1 <    \frac{7 \varepsilon}{29} ,
   	$$
   	$$
   	\supp g_1 \cap  \supp g_2 = \varnothing  \sem \text{and} \sem \Vert g_1+g_2 \Vert _1  =1.
   	$$
   	As a consequence,   we have that 
   	\begin{equation}
   	\label{c0-S-A}
   	\Vert  S (\chi_A) \chi _{ \Omega  \backslash \supp g_1} \Vert _1 =  \Vert  (g_1 - S(\chi_A))  \chi _{ \Omega \backslash \supp g_1} \Vert _1 \le \Vert g_1 -  S (\chi_A)  \Vert _1 < \frac{7 \varepsilon}{29}
   	\end{equation}
   	and also
   	\begin{equation}
   	\label{c0-S-B}
   	\Vert  S (\chi_B) \chi _{ \Omega \backslash \supp g_2} \Vert _1  < \frac{7 \varepsilon}{29}.
   	\end{equation}
   	
   	We define the operator $U: c_0 \llll L_1 (\mu)$ by 
   	$$
   	U(x)= S(x \chi_A) \chi _{\supp g_1} +  S(x \chi_B) \chi _{\supp g_2} \seg (x \in c_0).
   	$$
   	The operator $U$ is linear, bounded and  positive.   Since $U(x)= U(x \chi _{A \cup B})$ for any element $x \in c_0$  and $A\cup B$ is finite  we obtain that 
   	$$
   	\Vert  U \Vert =  \Vert U(\chi_{A \cup B})  \Vert _1  = \Vert S(\chi_A) \chi _{\supp g_1} +  S(\chi_B) \chi _{\supp g_2}  ) \Vert _1 \le  \Vert S \Vert = 1.
   	$$
   	Now we estimate the  distance between $U$ and $S$. For an element  $x \in B_{ c_0}$ it is satisfied  
   	\begin{align*}
   	\Vert (U -  S)(x) \Vert _1 & =  \Vert  S(x \chi_A) \chi _{\supp g_1} +  S(x \chi_B) \chi _{\supp g_2}   ) - S(x) \Vert _1 \\
   	&  = \Vert  S(x \chi_A) \chi _{\supp g_1} +  S(x \chi_B) \chi _{\supp g_2}   ) - S(x \chi_A)  - S(x \chi_B)  -S(x \chi_C) \Vert _1 \\
   	&\le \Vert  S(x \chi_A) \chi _{ \Omega  \backslash \supp g_1}  \Vert _1 +  \Vert  S(x \chi_B) \chi _{\Omega \backslash\supp g_2}   )  \Vert _1 + \Vert S(x \chi_C) \Vert _1 \\
   	&\le \Vert  S(\chi_A) \chi _{ \Omega  \backslash \supp g_1}  \Vert _1 +  \Vert  S( \chi_B) \chi _{\Omega  \backslash\supp g_2}   )  \Vert _1 + \Vert S(x \chi_C) \Vert _1 \\
   	&<  \frac{ 14 \varepsilon}{29} + \eta  < \frac{\varepsilon}{2} \sem \text{(by \eqref{c0-S-A},  \eqref{c0-S-B} and \eqref{SxC-small})}.
   	\end{align*}
   	We proved that $\Vert U - S \Vert  < \frac{\varepsilon}{2}$  and so $\Vert U \Vert \ge 1 - \dfrac{\varepsilon}{2} > 0$. 
   	Since $x_0 \in S_{c_0}$ the element  $u_0$ given by $u_0 = \chi_B- \chi_A + x_0 \chi _C \in  S_{ c_0}$ and satisfies
   	$$
   	\Vert u_0 - x_0 \Vert \le \eta < \varepsilon.
   	$$
   	Since $g_1$ and $g_2$ have disjoint supports we also have that
   	\begin{eqnarray*}
   		\Vert U (u_0) \Vert _1 & = \Vert S ( -\chi_A) \chi_{\supp g_1} +  S ( \chi_B) \chi_{\supp g_2} \Vert _1  \\
   		& =  \Vert S ( \chi_A) \chi_{\supp g_1}  +   S ( \chi_B) \chi_{\supp g_2} \Vert _1  \\
   		&  = \Vert U( \chi_{A  \cup B})\Vert _1 = \Vert U \Vert  .
   	\end{eqnarray*}
   	If we take $T= \frac{U}{\Vert U \Vert}$, the operator $T \in S_{L(c_0, L_1 (\mu))}$, is a positive operator, attains its norm at $u_0$ and satisfies that
   	$$
   	\Vert T - S \Vert \le  \Vert T - U \Vert +\Vert  U-S \Vert  =  \bigl\vert 1 - \Vert U \Vert  \bigr \vert   +\Vert  U-S \Vert \le 2 \Vert U - S \Vert < \varepsilon.
   	$$
   	We proved that the pair $(c_0, L_1 (\mu))$  has the Bishop-Phelps-Bollobás property for positive operators.  Notice that in case that $x_0 $ is positive,  the element  $u_0$  is also positive.
   \end{proof}
   %} 
   
Lastly we provide an example showing that the property that we considered is not trivial.  

\begin{example}
Let $Y=c_0$ as a set, endowed with the norm given by
$$
\tnorm x \tnorm = \Vert x \Vert + \Bigl\Vert \Bigl\{  \frac{ x_n}{2^n}\Bigr\} \Bigr \Vert _2 \seg (x \in c_0),
$$
where $\Vert \ \Vert$  is the usual norm of $c_0$. Then the pair $(c_0,Y)$ does not satisfy the  Bishop-Phelps-Bollobás property for positive operators. 
\end{example}
\begin{proof}
It is clear that $\nnn$ is a norm equivalent to the usual norm  of $c_0$ and it is a lattice norm on $Y$. Also the space $Y$ is strictly convex. So the formal identity from $c_0$ to $Y$ cannot be approximated by norm attaining operators by \cite[Proposition 4]{Lin}. Since the formal identity is a positive operator we are done.
\end{proof}

\vskip3mm

\textbf{Acknowledgement.} The authors kindly thank  to M. Masty{\l}o who suggested to study the property considered in this paper during a research stay in the University of Granada.

\bibliographystyle{amsalpha}

\begin{thebibliography}{99}
	%\bibliographystyle{amsplain}
	%\begin{thebibliography}{10}
	%	

\bibitem{AbAl}
 Y.A. Abramovich and C.D. Aliprantis,  \textit{An invitation to operator theory,} Graduate Studies in Mathematics, \textbf{50},  American Mathematical Society, Providence, RI, 2002.
%
	\bibitem{Acs}
	M.D. Acosta, \textit{Denseness of norm attaining mappings},  
	RACSAM. Rev. R. Acad. Cienc. Exactas Fís. Nat. Ser. A Mat.  {\bf 100} (2006), 9--30.
	%	
	\bibitem{AcB}
	M.D. Acosta,   {\it The Bishop-Phelps-Bollob{\'a}s property for operators on $C(K)$},  Banach J. Math. Anal. \textbf{10} (2016), 307--319. 
	%
	\bibitem{Acbc}
	M.D. Acosta, \textit{On the  Bishop-Phelps-Bollob{\'a}s property},  
	preprint, 2019.
	%	
	\bibitem{AAGM}
	M.D. Acosta, R.M. Aron, D. Garc\'{\i}a and M. Maestre, \textit{The Bishop-Phelps-Bollob\'{a}s
		theorem for operators}, J. Funct. Anal. \textbf{254} 
	%(11) 
	(2008), 2780--2799.
	%
	\bibitem{ABCC}
	M.D. Acosta, J. Becerra-Guerrero, Y.S. Choi, M. Ciesielski, S.K. Kim, H.J. Lee, M.L. Louren\c{c}o and M. Mart{\'\i}n, {\it The Bishop-Phelps-Bollob\'{a}s property for operators between spaces of continuous functions}, Nonlinear Anal. \textbf{95} (2014), 323--332.
	%
	\bibitem{ABGj}
	M.D. Acosta, J. Becerra-Guerrero, D. Garc{\'\i}a,  S.K. Kim and M. Maestre, {\it Bishop-Phelps-Bollob\'{a}s property for  certain spaces of operators}, J. Math. Anal. Appl. \textbf{414} (2014), 532--545.
	%
	\bibitem{ABGp}
	M.D. Acosta, J. Becerra-Guerrero, D. Garc{\'\i}a,  S.K. Kim and M. Maestre, {\it The Bishop-Phelps-Bollob\'{a}s property:  a finite-dimensional approach}, Publ. Res. Inst. Math. Sci. \textbf{51}  (2015), 173--190.
	%
	\bibitem{ABGt}
	M.D. Acosta, J. Becerra-Guerrero, D. Garc{\'\i}a  and M. Maestre, {\it The  Bishop-Phelps-Bollob\'{a}s theorem for bilinear forms}, Trans. Amer. Math. Soc.  \textbf{365} (2013), 5911--5932.
	%	

	\bibitem{AD}
	M.D. Acosta and J.L. D{\'a}vila, {\it  A basis of $\R ^n$ with  good isometric properties  and some applications to denseness  of norm attaining operators,}  preprint, {\tt arXiv:1811.08387v1 [math.FA]}.
	
	\bibitem{ADS}
	M.D. Acosta, J.L. D{\'a}vila and M. Soleimani-Mourchehkhorti, {\it  Characterization of the Banach spaces $Y$ satisfying that the pair $(\ell_\infty ^4, Y)$ has the Bishop-Phelps-Bollob\'{a}s property for operators},  J. Math. Anal. Appl. \textbf{470} (2019),  690--715.
	
	\bibitem{AlBur}
	C.D.  Aliprantis and O.  Burkinshaw,  \textit{Positive operators,} Springer, Dordrecht, 2006. 
	
	%	
	%\bibitem{AGKM}
	%M.D. Acosta,  D. Garc{\'\i}a,  S.K. Kim and M. %Maestre, {\it The Bishop-Phelps-Bollob{\'a}s %property for operators from $c_0$  into some %Banach spaces},  J. Math. Anal. Appl. 445 (%2017), 1188--1199. 
	%	
	\bibitem{ACK}
	R.M. Aron, B. Cascales  and O. Kozhushkina, {\it The Bishop-Phelps-Bollob{\'a}s theorem and Asplund operators}, Proc. Amer. Math. Soc. \textbf{139} (2011),  3553--3560.
	
	\bibitem{ACGM}
	R.M. Aron, Y.S. Choi,  D. Garc\'{\i}a and M. Maestre, {\it The Bishop-Phelps-Bollob\'{a}s theorem for ${\mathcal{L}}(L_1(\mu), L_\infty[0,1])$}, Adv. Math. \textbf{228} (2011),  617--628.
	
	\bibitem{BP} 
	E. Bishop and R.R. Phelps, {\it A proof that every Banach space is subreflexive}, Bull. Amer. Math. Soc. \textbf{67} (1961), 97--98.
	%	
	\bibitem{Bol}
	B. Bollob\'as, {\it An extension to the theorem of Bishop and Phelps}, Bull. London. Math.
	Soc. {\bf 2} (1970), 181--182.
	%	
	\bibitem{BoDu}
	F.\,F.~Bonsall and J.~Duncan, {\it  Numerical ranges  II,} London
	Mathematical Society Lecture Notes Series, No. {\bf 10}, Cambridge
	University Press, New York-London, 1973.
	%
	\bibitem{CGK}
	B. Cascales, A.J. Guirao and V. Kadets, {\it A Bishop-Phelps-Bollob{\'a}s type theorem for uniform algebras},  Adv. Math. \textbf{240} (2013), 370--382. 
	%
	\bibitem{CKMMR}
	M. Chica, V. Kadets, M. Mart{\'\i}n, S. Moreno-Pulido and F. Rambla-Barreno, {\it Bishop-Phelps-Bollob\'{a}s moduli of a Banach space}, J. Math.  Anal. Appl. {\bf 412}
	(2014), 
	%no. 2, 
	697--719.
	%
	\bibitem{CK1}
	Y.S. Choi and S.K.  Kim, {\it The Bishop-Phelps-Bollob\'as theorem for operators from $L_1(\mu)$ to Banach spaces with the Radon-Nikod\'{y}m property}, J. Funct. Anal.   \textbf{261} (2011), 1446--1456.
	%
	\bibitem{CKLj}
	Y.S. Choi, S.K.  Kim, H.J. Lee and M. Mart{\'i}n, {\it The Bishop-Phelps-Bollob{\'a}s theorem for operators on $L_1 (\mu)$}, J. Funct. Anal.  \textbf{267} (2014),  214--242. 
	%	
	\bibitem{KiI}
	S.K. Kim, {\it The Bishop-Phelps-Bollob{\'a}s theorem for operators from $c_0$ to uniformly convex spaces}, Israel J. Math.  \textbf{197} (2013), 425--435.
	
	\bibitem{KLc}
	S.K. Kim and H.J. Lee, {\it Uniform convexity and Bishop-Phelps-Bollob{\'a}s property},
	Canad. J. Math. \textbf{66} (2014),  373--386. 
	
	\bibitem{KL2}
	S.K.  Kim and  H.J. Lee, {\it The Bishop-Phelps-Bollob{\'a}s property for operators from $C(K)$ to uniformly convex spaces},  J. Math. Anal. Appl. \textbf{421} (2015), 51--58. 
	%
	\bibitem{KLL}
	S.K.  Kim, H.J. Lee and P.K. Lin, {\it The Bishop-Phelps-Bollob{\'a}s property for operators from $L_\infty (\mu)$ to uniformly convex Banach spaces}, J. Nonlinear Convex Anal. \textbf{17} (2016), 243--249.
	%
\bibitem{Lin}
J. Lindenstrauss, On operators which attain their norm, Israel J. Math.  \textbf{1} (1963)
139--148.

\bibitem{Scha} 
W. Schachermayer,  \textit{Norm attaining operators on some classical Banach spaces,}  Pacific J. Math.  \textbf{105}  (1983),  427--438.

	
\end{thebibliography}

\end{document}